\documentclass[12pt]{amsart}

\usepackage{amsmath,amsthm}
\usepackage{amsfonts}
\usepackage{verbatim}
\usepackage{amssymb}
\usepackage{txfonts}
\usepackage{amscd}
\usepackage{bbm}
\usepackage{hyperref}
\usepackage{fouriernc}

\usepackage{tikz-cd}
\usepackage{gauss}
\usepackage{cleveref}
\usepackage{mathrsfs}

\usepackage{tikz}

\usepackage{geometry}

\usepackage[alphabetic]{amsrefs}
\DeclareMathAlphabet{\altmathcal}{OMS}{cmsy}{m}{n}

\newtheorem{theorem}{Theorem}[section]
\newtheorem{corollary}[theorem]{Corollary}
\newtheorem{lemma}[theorem]{Lemma}

\newtheorem{definition}[theorem]{Definition}

\newtheorem{remark}[theorem]{Remark}
\newtheorem{example}[theorem]{Example}

\newcommand{\RR}{\mathbb{R}}
\newcommand{\ZZ}{\mathbb{Z}}

\newcommand{\Aut}{\operatorname{Aut}}
\newcommand{\SAut}{\operatorname{SAut}}

\newcommand{\Sym}{\operatorname{Sym}}

\newcommand{\SL}{\operatorname{SL}}
\newcommand{\Sq}{\operatorname{Sq}}
\newcommand{\Adj}{\operatorname{Adj}}
\newcommand{\Op}{\operatorname{Op}}
\newcommand{\Free}{\operatorname{Free}}
\newcommand{\Edg}{\operatorname{Edg}}
\newcommand{\Tri}{\operatorname{Tri}}
\newcommand{\Tet}{\operatorname{Tet}}
\newcommand{\Stab}{\operatorname{Stab}}

\DefineSimpleKey{bib}{primaryclass}{}
\DefineSimpleKey{bib}{archiveprefix}{}

\BibSpec{arXiv}{%
	+{}{\PrintAuthors}{author}
	+{,}{ \textit}{title}
	+{}{ \parenthesize}{date}
	+{,}{ arXiv }{eprint}
}

\author{Piotr Mizerka}
\email{pmizerka@impan.pl}
\address{Institute of Mathematics of the Polish Academy of Sciences, \'{S}niadeckich 8, 00-656, Warsaw, Poland}

\begin{document}
	\sloppy 
	
	\title[Inducing spectral gaps for $\SL_n(\ZZ)$ and $\SAut(F_n)$]{Inducing spectral gaps for the cohomological\\ Laplacians of $\SL_n(\ZZ)$ and $\SAut(F_n)$}
	
	\begin{abstract}
		The technique of inducing spectral gaps for cohomological Laplacians in degree zero was used by Kaluba, Kielak and Nowak \cite{kkn} to prove property (T) for $\Aut(F_n)$ and $\SL_n(\ZZ)$. In this paper, we adapt this technique to Laplacians in degree one. This allows to provide a lower bound for the cohomological Laplacian in degree one for $\SL_n(\ZZ)$ for every unitary representation. In particular, one gets in that way an alternative proof of property (T) for $\SL_n(\ZZ)$ whenever $n\geq 3$.
	\end{abstract}
	\maketitle

	\section{Introduction}
	Cohomological Laplacians were introduced by Bader and Nowak \cite{bader-nowak} in the context of vanishing of group cohomology with unitary coefficients. Such vanishing of the first cohomology is known to be equivalent to Kazhdan's property (T) and an algebraic characterization for it was found by Ozawa \cite{ozawa1} for finitely generated groups. It turns out that property (T) can be described in terms of a specific degree zero \emph{Laplacian} $\Delta$: Ozawa showed that a group $G$ has property (T) if and only if $\Delta^2-\lambda\Delta$ is a sum of squares in the real group ring $\mathbb{R}G$ for some positive spectral gap $\lambda$. Using this characterization, Kaluba, Kielak, Nowak, and Ozawa \cite{kno,kkn} were able to show property (T) for $\SL_n(\ZZ)$ and $\Aut(F_n)$ for $n\geq 3$ and $n\geq 5$ respectively.  Estimating the spectral gap $\lambda$ from below has another advantage of prividing a lower bound for $\kappa(G,S)$, the \emph{Kazhdan constant} associated to a generating set $S$ of a group $G$. It is known that $G$ has property (T) if and only if $\kappa(G,S)>0$ and this inequality does not depend on the generating set. For a symmetric generating set $S=S^{-1}$, the lower bound for Kazhdan constants is given by the formula (see \cite[Remark 5.4.7]{bekka} and \cite{kn2018})
	$$
	\sqrt{\frac{2\lambda}{|S|}}\leq\kappa(G,S),
	$$
	where $\lambda$ is a positive constant such that $\Delta^2-\lambda\Delta$ is a sum of squares.
	
	We focus on providing lower bounds for the spectral gap of $\Delta_1$, the one-degree Laplacian ($\Delta_1$ can be in fact any group ring matrix which can be used to compute the first group cohomology -- details can be found in \cite{bader-nowak} -- and we focus on its particular model here, defined by \emph{Fox derivatives} \cite{fox1,fox2,lyndon1}). The existence of such a positive spectral gap for a group $G$ is equivalent  to vanishing of the first cohomology of $G$ with unitary coefficients and \emph{reducibility} of the second (reducibility means that the images of the differentials in the chain complex of Hilbert spaces computing the cohomology for a given unitary representation are closed), see \cite[The Main Theorem]{bader-nowak}. Once a positive spectral gap is found for $\Delta_1$, it turns out to be a lower bound for the Ozawa's spectral gap, see Remark \ref{remark:star_conj}. We acknowledge that it has been recently announced by U. Bader and R. Sauer that the reducibility condition is a consequence of vanishing of cohomology one degree lower, see \cite{Bader_Sauer}. Thus, $\Delta_1$ having a positive spectral gap is equivalent to a group having property (T). By \cite{bader-nowak} and Theorem 3.11 from \cite{Bader_Sauer} and the paragraph below it, we know actually that a more general statement holds: for any $n\geq 1$, the existence of a spectral gap for some model of $\Delta_n$ is equivalent to the vanishing of the $n$-th cohomology of a group with unitary coefficients.  
	
	Recently, M. Kaluba, P. W. Nowak and the author showed that for $\Delta_1$, the first cohomological Laplacian of $\SL_3(\ZZ)$, and $\lambda=0.32$, the matrix $\Delta_1-\lambda I$ is a sum of squares (\cite[Theorem 1.1]{kmn}). In this paper we show how to use a slightly stronger statement for $\SL_3(\ZZ)$ to induce the spectral gap for $G_n=\SL_n(\ZZ)$, whenever $n\geq 3$. Our technique works as well for the case $G_n=\SAut(F_n)$. We remark however that in such a case we were unable to obtain such a spectral gap for any particular $n$. On the other other hand, in his preprint \cite{nitsche} M. Nitsche shows property (T) for $\SAut(F_4)$. This, in combination with the work of Bader and Nowak \cite{bader-nowak} and the announcement of Bader and Sauer, would yield the desired gap for the Laplacian in degree one for $\SAut(F_4)$. Nevertheless, this will still not allow us to induce the spectral gap since we would need a spectral gap for a specific summand of the Laplacian of $\SAut(F_4)$. 
	
	Our induction technique is inspired by the idea of decomposing the generating set of $G_n$ into three parts: square, adjacent, and opposite, see \cite{kkn}. This allows us to decompose $\Delta_1$ for any $G_n$ into the sum of the appropriate parts. Denoting the adjacent part by $\Adj_n$ for any $n\geq 3$, and the generating set for $G_n$ by $\altmathcal{S}_n$, the main theorem can be stated as follows.
	\begin{theorem}[cf. \Cref{corollary:main_ind}]\label{theorem:main0}
		Suppose $\Adj_m-\lambda I_{|\altmathcal{S}_m|}$ is a sum of squares in $\mathbb{M}_{|\altmathcal{S}_m|\times |\altmathcal{S}_m|}(\RR G_m)$. Then \begin{align*}
			\Adj_n-\frac{n-2}{m-2}\lambda I_{|\altmathcal{S}_n|}
		\end{align*}
		is a sum of squares in $\mathbb{M}_{|\altmathcal{S}_n|\times |\altmathcal{S}_n|}(\RR G_n)$ for any $n\geq m$.
	\end{theorem}
	\noindent Bounding the spectral gap for the adjacent part $\Adj_3$ of $\SL_3(\ZZ)$, we get the following
	\begin{corollary}[cf. \Cref{corollary:main}]\label{corollary:main_intr}
		Let $\Delta_1$ be the cohomological Laplacian in degree one of $G_n=\SL_n(\ZZ)$. Then $\Delta_1-\lambda I_{|\altmathcal{S}_n|}$ is a sum of squares for $\lambda=0.217(n-2)$.
	\end{corollary}
	\noindent To our knowledge, the lower bounds for the spectral gaps we obtained are the first examples of explicit computations for $\Delta_1$ for an infinite family of property (T) groups.
	
	This article is organized as follows. In \cref{section:fox} we introduce Fox derivatives which provide a setting to calculate one-degree Laplacians and define the positivity in the group ring setting in terms of sums of squares. In the next section we introduce the presentions for $G_n$: using the \emph{elementary matrices} for $G_n=\SL_n(\ZZ)$ and \emph{Nielsen transvections} for $G_n=\SAut(F_n)$. Next, we decompose the Laplacian in degree one into the three parts mentioned above. In \cref{section:symmetrization} we show \Cref{theorem:main0} using the technique of symmetrization and apply this result in \cref{section:application_sln} to obtain a lower bound for the spectral gap of $\Delta_1$ for $G_n=\SL_n(\ZZ)$ whenever $n\geq 3$.
	
	\subsection*{Acknowledgements}
	The author would like to express his sincere thanks to Prof. Dawid Kielak for many valuable suggestions concerning the induction technique, in particular for pointing out the decomposition of $\Delta_1^+$ defined in \cref{section:decompositions}. I am also greatly indebted to Dr Marek Kaluba for his collaboration on the substance of \cref{subsection:general_setting,subsection:inv} and the suggestion to express the symmetrization used in the induction in the language involving simplices, as well as his careful reading of the whole contents and many valuable remarks.
	
	This research was supported by National Science Centre, Poland SONATINA 6 grant \emph{Algebraic spectral gaps in group cohomology} (grant agreement no. 2022/44/C/ST1/00037).
	
	\section{Fox derivatives and sums of squares}\label{section:fox}
	We describe the model of $\Delta_1$ which appears in Theorem \ref{theorem:main0}. This can be achieved by means of \emph{Fox derivatives}. Further details can be found in the papers of Fox and Lyndon, \cite{fox1,fox2,lyndon1,lyndon2}.
	
	Let $G=\langle s_1,\ldots,s_n|r_1,\ldots,r_m\rangle$ be a finitely presented group. The Fox derivatives $\frac{\partial}{\partial s_j}$ are the endomorphisms of $\ZZ F_n$, the group ring of the free group $F_n=\langle s_1,\ldots,s_n\rangle$, given by the following conditions:
	\begin{align*}
		\frac{\partial 1}{\partial s_j}&=0,\\
		\frac{\partial s_i}{\partial s_j}&=\delta_{ij},\\
		\frac{\partial(uv)}{\partial s_j}&=\frac{\partial u}{\partial s_j}+u\frac{\partial v}{\partial s_j},
	\end{align*}
	where $\delta_{ij}$ denotes the Kronecker delta. The \emph{Jacobian} of $G$ is then the matrix $J=\left[\frac{\partial r_i}{\partial s_j}\right]\in\mathbb{M}_{m\times n}(\ZZ G)$ obtained by quotienting the Fox derivatives into the group ring $\ZZ G$. Denoting by $d$ the vertical vector $\left[1-s_j\right]\in\mathbb{M}_{n\times 1}(\ZZ G)$, the one-degree Laplacian $\Delta_1$ associated to the presentation $G=\langle s_1,\ldots,s_n|r_1,\ldots,r_m\rangle$ is defined by
	$$
	\Delta_1=\Delta_1^++\Delta_1^-,
	$$
	where $\Delta_1^+=J^*J$ and $\Delta_1^-=dd^*$, and, for any matrix $M$ with entries in $\ZZ G$, the matrix $M^*$ is the composition of the matrix transposition with the $*$-involution on $\ZZ G$ given by $g\mapsto g^{-1}$.
	\begin{definition}
		We say that a matrix $M\in\mathbb{M}_{k\times k}(\RR G)$ is a \textbf{sum of squares} if there exist matrices $M_1,\ldots,M_l\in\mathbb{M}_{k\times k}(\RR G)$ such that $M=M_1^*M_1+\ldots M_l^*M_l$.
	\end{definition}
	\noindent It follows from \cite[Lemma 14]{bader-nowak} that $\Delta_1^{\pm}$, and therefore $\Delta_1$ as well, are sums of squares.
	
	By saying that a matrix $M\in\mathbb{M}_{k\times k}(\RR G)$ has a positive \emph{spectral gap} we mean that there exists a positive $\lambda$ such that $M-\lambda I_k$ is a sum of squares.
	
	As noted in \cite[Lemma 2.1]{kmn}, it turns out that we have some freedom in choosing the relations which build up $\Delta_1^+$. Denote 
	$$
	J_i=\left[\frac{\partial r_i}{\partial s_1}\cdots\frac{\partial r_i}{\partial s_n}\right]
	$$ 
	and put $\Delta_{I}^+=\sum_{i\in I}J_i^*J_i$ for any subset of indices $I\subseteq\{1,\ldots,m\}$. Then, once $\Delta_{I}^++\Delta_1^-\in\mathbb{M}_{n\times n}(\RR G)$ has a spectral gap then $\Delta_1=\Delta_{I}^++\Delta_{\{1,\ldots,m\}\setminus I}^++\Delta_1^-\in\mathbb{M}_{n\times n}(\RR G)$ has at least the same spectral gap.
	
	Computing lower bounds for sepctral gaps for the first Laplacians provides the bounds for Ozawa's spectral gap:
	\begin{remark}\label{remark:star_conj}
		\emph{Suppose the matrix $C\Delta_{I}^++\Delta_1^--\lambda I$ is a sum of squares for some subset of relator indices $I\subseteq\{1,\ldots,m\}$ and a positive constant $C$ and let $\Delta=d^*d$. Then $\Delta^2-\lambda\Delta=d^*\left(C\Delta_{I}^++\Delta_1^--\lambda I\right)d$ is a sum of squares as well (note that $d^*\Delta_{I}^+d$ is always zero). Note that in the case the generating set $\altmathcal{S}$ and the set of its inverses, $\altmathcal{S}^{-1}$, are disjoint, then $\Delta$ is the zero degree Laplacian $|\altmathcal{S}|-\sum_{s\in\altmathcal{S}}s$ for the symmetric generating set $\altmathcal{S}\cup \altmathcal{S}^{-1}$. }
	\end{remark}
	
	\section{Presentations of $\SL_n(\ZZ)$ and $\SAut(F_n)$}
	In this section, we introduce the presentations of $\SL_n(\ZZ)$ and $\SAut(F_n)$ which we shall use to perform the Fox calculus and prove property (T) for these groups. The presentations can be found in \cite{sln_pres} and \cite{Gersten} for $\SL_n(\mathbb{Z})$ and $\SAut(F_n)$ respectively. For the latter, we provide an equivalent presentation which is easy to obtain from \cite{Gersten}.
	
	The presentation for $\SL_n(\ZZ)$ is given below.
	\begin{align*}
		\SL_n(\ZZ)=\big\langle E_{ij},1\leq i\neq j\leq n\big|&[E_{ij},E_{kl}],i\neq l, j\neq k, (i,j)\neq(k,l),\\
		&[E_{ij},E_{jk}]E_{ik}^{-1}\text{ for distinct }i,j,k,\\
		&\left(E_{12}E_{21}^{-1}E_{12}\right)^4\big\rangle.
	\end{align*}
	The generator $E_{ij}$ represents the elementary matrix with $1$ at the $(i,j)$-th entry and the diagonal and zeroes elsewhere. Using \Cref{remark:star_conj}, we shall exclude the relation $\left(E_{12}E_{21}^{-1}E_{12}\right)^4$ from further considerations.
	
	The generators of $\SAut(F_n)$ are denoted by $\lambda_{ij}$ and $\rho_{ij}$ for $1\leq i\neq j\leq n$ and are called left and right \emph{Nielsen transvections} respectively. The relator set contains commutators and pentagonal relations as well, along with other relations which we shall not consider, again by \Cref{remark:star_conj}. We only present the commutator and pentagonal relations therefore. These relations are defined for distinct indices $i,j,k,l\in\{1,\ldots,n\}$ as follows:
	\begin{gather*}
		[\lambda_{i,j},\rho_{ij}],\quad [\lambda_{ij},\lambda_{kl}],\quad [\rho_{ij},\rho_{kl}],
		\quad [\lambda_{ij},\rho_{kl}],\\
		\quad [\lambda_{ij},\lambda_{kj}],\quad [\rho_{ij},\rho_{kj}],
		\quad [\lambda_{ij},\rho_{ik}],\quad [\lambda_{i,j},\rho_{kj}],\\
		\quad \lambda_{ik}^{\pm}[\lambda_{jk}^{\pm},\lambda_{ij}^{-1}],
		\quad \rho_{ik}^{\pm}[\rho_{jk}^{\pm},\rho_{ij}^{-1}],
		\quad \lambda_{ik}^{\pm}[\rho_{jk}^{\mp},\lambda_{ij}],
		\quad \rho_{ik}^{\pm}[\lambda_{jk}^{\mp},\rho_{ij}].
	\end{gather*}
	The transvection generators correspond in $\SAut(F_n)$ to the following automorphisms of $F_n=\langle s_1,\ldots,s_n\rangle$\footnote{We take in $\SAut(F_n)$ the group operation order given by composition of automorphisms, i.e. $\alpha*\beta=\alpha\circ\beta$. Note that this is different than the convention of Gersten \cite{Gersten}. In his paper the convention is: $\alpha*\beta=\beta\circ\alpha$. We had to reverse the order in pentagonal relators for that account.}: 
	\begin{gather*}
		\lambda_{ij}(s_k) =
		\begin{cases}
			s_js_i & \text{ if $k=i$,} \\
			s_k & \text{otherwise,}
		\end{cases}\quad
		\rho_{ij}(s_k) =
		\begin{cases}
			s_is_j & \text{ if $k=i$,} \\
			s_k & \text{otherwise.}
		\end{cases}
	\end{gather*}
	\section{Decomposition into square, adjacent and opposite part}\label{section:decompositions}
	Let $n\geq 3$ be an integer and $G_n$ be $\SL_n(\ZZ)$ or $\SAut(F_n)$. We introduce the following decompositions of the Laplacians $\Delta_1^{\pm}$ of $G_n$:
	\begin{eqnarray*}
		\Delta_1^{\pm} =\Sq_n^{\pm}+\Adj_n^{\pm}+\Op_n^{\pm}.
	\end{eqnarray*}
	Our decomposition is inspired by \cite[pp. 543 -- 545]{kkn}.
	
	The generators of $G_n$ take form $\omega_{ij}$ for $1\leq i\neq j\leq n$ and $\omega$ denoting either elementary matrices or Nielsen transvections. Then, for any $3\leq m\leq n$, there exists a natural embedding $G_m\hookrightarrow G_n$ given by $\omega_{ij}\mapsto\omega_{ij}$. From now on, when talking about the relations of $G_n$, we shall restrict ourselves to commutator and pentagonal relations only.
	
	Let $C_n$ be the $(n-1)$-simplex with the vertices $\{1,\ldots,n\}$ and let $F(C_n)$ denote the set of its faces, that is, the subsets of the set $\{1,\ldots,n\}$. Denoting by $\altmathcal{S}_n$ the generator set of $G_n$ and by $\Free(\altmathcal{S}_n)$ the free group on $\altmathcal{S}_n$, one can define the map
	\begin{align*}
		\phi:\Free(\altmathcal{S}_n)\longrightarrow F(C_n)
	\end{align*}
	by sending an element of $\Free(\altmathcal{S}_n)$ to the face on the indices of the generators occurring in that element, i.e. 
	$$
	\phi(\omega_{i_1j_1}\ldots\omega_{i_kj_k})=\{i_1,j_1,\ldots,i_k,j_k\}.
	$$
	The generators $\omega_{i_lj_l}$ can be either elementary matrices or left or right Nielsen transvections.
	
	In order to define $\Sq_n^{\pm}$, $\Adj_n^{\pm}$ and $\Op_n^{\pm}$ we distinguish from $F(C_n)$ the following three subsets of faces
	\begin{align*}
		\Edg_n&=\{f\in F(C_n)||f|=2\},\\
		\Tri_n&=\{f\in F(C_n)||f|=3\},\\
		\Tet_n&=\{f\in F(C_n)||f|=4\},
	\end{align*}
	that is, $\Edg_n$, $\Tri_n$ and $\Tet_n$ constitute respectively the sets of edges, triangles, and tetrahedra of $C_n$.
	
	The decomposition $\Delta_1^-=\Sq_n^-+\Adj_n^-+\Op_n^-$ is obtained as follows:
	\begin{align*}
		\left(\Sq^-_n\right)_{s,t}&=\left(\Delta^-_1\right)_{s,t}\text{ for }\phi(s)\cup\phi(t)\in\Edg_n\text{ and } 0 \text{ otherwise},\\
		\left(\Adj^-_n\right)_{s,t}&=\left(\Delta^-_1\right)_{s,t}\text{ for }\phi(s)\cup\phi(t)\in\Tri_n\text{ and } 0 \text{ otherwise},\\
		\left(\Op^-_n\right)_{s,t}&=\left(\Delta^-_1\right)_{s,t}\text{ for }\phi(s)\cup\phi(t)\in\Tet_n\text{ and } 0 \text{ otherwise}.
	\end{align*}
	We remark that the group ring elements $d^*\Sq^-_nd$, $d^*\Adj^-_nd$, and $d^*\Op^-_nd$ are equal to the elements $\Sq_n$, $\Adj_n$, and $\Op_n$ as defined in \cite{kkn}.
	
	The decomposition $\Delta_1^+=\Sq^+_n+\Adj^+_n+\Op^+_n$ is defined, in turn, by summing over specific relators of $G_n$.
	\begin{align*}
		\left(\Sq^+_n\right)_{s,t}&=\sum_{\phi(r)\in\Edg_n}\left(\frac{\partial r}{\partial s}\right)^*\frac{\partial r}{\partial t},\\
		\left(\Adj^+_n\right)_{s,t}&=\sum_{\phi(r)\in\Tri_n}\left(\frac{\partial r}{\partial s}\right)^*\frac{\partial r}{\partial t},\\
		\left(\Op^+_n\right)_{s,t}&=\sum_{\phi(r)\in\Tet_n}\left(\frac{\partial r}{\partial s}\right)^*\frac{\partial r}{\partial t}.
	\end{align*}
	Note that $\left(\Adj^+_n\right)_{s,t}$ and $\left(\Op^+_n\right)_{s,t}$ can be non-zero only if $s=t$ or $\phi(s)\cup\phi(t)$ belongs to $\Tri_n$ and $\Tet_n$ respectively. Since we do not consider relation $\left(E_{12}E_{21}^{-1}E_{12}\right)^4$ for $\SL_n(\ZZ)$, there are no relators $r$ such that $\phi(r)\in\Edg_n$ for this case and so $\Sq_n^+$ is the zero matrix for $\SL_n(\ZZ)$. This is not the case for $\SAut(F_n)$ due to the presence of relations $[\lambda_{i,j},\rho_{ij}]$ - this is, however, not a problem for us, since we only need $\Sq_n^+$ to be a sum of squares which is obvious by definition.
	
	 We shall call $\Sq_n^{\pm}$, $\Adj_n^{\pm}$ and $\Op_n^{\pm}$ the \emph{square}, \emph{adjacent} and \emph{opposite} part respectively.
	
	\begin{example}
		\emph{Let $n=3$ and let us focus on $G_3=\SL_3(\mathbb{Z})$. We have six generators of $G_3$, namely the elementary matrices $E_{12}$, $E_{21}$, $E_{13}$, $E_{31}$, $E_{23}$, and $E_{32}$. The relator set $R$ we take into account contains the following $18$ relators:
			\begin{eqnarray*}
				&&[E_{12},E_{13}],[E_{13},E_{12}],[E_{12},E_{32}],[E_{32},E_{12}],\\
				&&[E_{13},E_{23}],[E_{23},E_{13}],[E_{21},E_{23}],[E_{23},E_{21}],\\
				&&[E_{21},E_{31}],[E_{31},E_{21}],[E_{31},E_{32}],[E_{32},E_{31}],\\
				&&[E_{12},E_{23}]E_{13}^{-1},[E_{13},E_{32}]E_{12}^{-1},[E_{21},E_{13}]E_{23}^{-1},\\
				&&[E_{23},E_{31}]E_{21}^{-1},[E_{31},E_{12}]E_{32}^{-1},[E_{32},E_{21}]E_{31}^{-1}.
			\end{eqnarray*}
			The simplex $C_3$ is a triangle. Its faces contain three edges -- more precisely, $\Edg_3=\{\{1,2\},\{1,3\},\{2,3\}\}$, and only one triangle, i.e. $\Tri_3=\{\{1,2,3\}\}$. There are no tetrahedra in $C_3$. 
			\begin{figure}[h]
				\centering
				\begin{tikzpicture}
					\coordinate (A) at (0,0);
					\coordinate (B) at (3,0);
					\coordinate (C) at (1.5,2.6);
					\draw[black, ultra thick] (A) -- (B);
					\draw[black, ultra thick] (A) -- (C);
					\draw[black, ultra thick] (B) -- (C);
					\filldraw[black] (A) circle (2.75pt) node[below left]{$1$};
					\filldraw[black] (B) circle (2.75pt) node[below right]{$2$};
					\filldraw[black] (C) circle (2.75pt) node[above]{$3$};
					\fill[black, opacity=0.1] (A) -- (B) -- (C) -- cycle;
					\coordinate (e13) at (-1.5,1.5);
					\coordinate (e23) at (4.5,1.5);
					\coordinate (e12) at (1.5,-1.5);
					\node at (e13) {$\{E_{13},E_{31}\}$};
					\node at (e23) {$\{E_{23},E_{32}\}$};
					\node at (e12) {$\{E_{12},E_{21}\}$};
					\draw[->,thick] (-0.65,1.5) -- (0.65,1.3);
					\draw[->,thick] (3.65,1.5) -- (2.35,1.3);
					\draw[->,thick] (1.5,-1.25) -- (1.5,-0.1);
				\end{tikzpicture}
				\caption{The simplex $C_3$ with the assignments of the generators of $\SL_3(\mathbb{Z})$ to its edges given by $\phi$.}
				\label{figure:sl3}
		\end{figure}
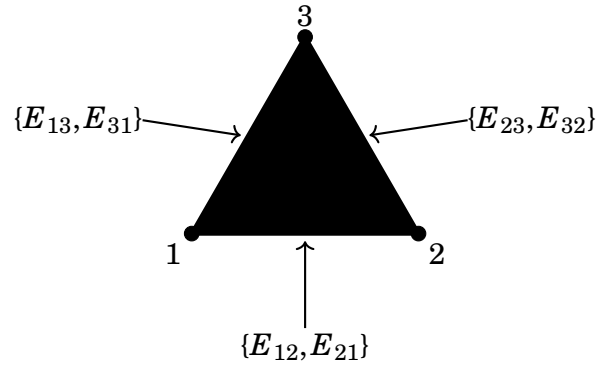}
		
		\noindent
		\emph{The function $\phi$ takes the following values on the generators (the situation depicted in \Cref{figure:sl3}):
			\begin{eqnarray*}
				\phi(E_{12})=\phi(E_{21})=\{1,2\},\quad\phi(E_{13})=\phi(E_{31})=\{1,3\},\quad\phi(E_{23})=\phi(E_{32})=\{2,3\},
			\end{eqnarray*}
			and is constant on the relators -- for every $r\in R$ one has $\phi(r)=\{1,2,3\}$. Thus, denoting by $(i,j)(k,l)$ the element $(1-E_{ij})(1-E_{kl})^*$, we get
			\begin{eqnarray*}
				\Sq_3^-=\begin{bmatrix}
					(1,2)(1,2)&0&(1,2)(2,1)&0&0&0\\
					0&(1,3)(1,3)&0&0&(1,3)(3,1)&0\\
					(2,1)(1,2)&0&(2,1)(2,1)&0&0&0\\
					0&0&0&(2,3)(2,3)&0&(2,3)(3,2)\\
					0&(3,1)(1,3)&0&0&(3,1)(3,1)&0\\
					0&0&0&(3,2)(2,3)&0&(3,2)(3,2)
				\end{bmatrix}
			\end{eqnarray*}
			and
			\begin{eqnarray*}
				\Adj_3^-=\begin{bmatrix}
					0&(1,2)(1,3)&0&(1,2)(2,3)&(1,2)(3,1)&(1,3)(3,2)\\
					(1,3)(1,2)&0&(1,3)(2,1)&(1,3)(2,3)&0&(1,3)(3,2)\\
					0&(2,1)(1,3)&0&(2,1)(2,3)&(2,1)(3,1)&(2,1)(3,2)\\
					(2,3)(1,2)&(2,3)(1,3)&(2,3)(2,1)&0&(2,3)(3,1)&0\\
					(3,1)(1,2)&0&(3,1)(2,1)&(3,1)(2,3)&0&(3,1)(3,2)\\
					(3,2)(1,2)&(3,2)(1,3)&(3,2)(2,1)&0&(3,2)(3,1)&0
				\end{bmatrix}.
			\end{eqnarray*}
			It is apparent that $\Op_3^-$ is the zero matrix. Since $\phi$ equals the face $\{1,2,3\}$ on all relators, we have $$
			\left(\Adj_3^+\right)_{s,t}=\sum_{r\in R}\left(\frac{\partial r}{\partial s}\right)^*\left(\frac{\partial r}{\partial t}\right)=\left(\Delta_1^+\right)_{s,t}
			$$ for any two generators $s,t$. In particular, $\Sq_3^+$ and $\Op_3^+$ are the zero matrices.}
	\end{example}
	
	\begin{example}
		\emph{This time, let $n=4$ and $G_4=\SAut(F_4)$. The matrices get significantly bigger for this case and for this account we shall restrict ourselves to the demonstration of the function $\phi$ only. The following $24$ transvections consitute the generators of $G_4$:
			\begin{eqnarray*}
				&&\lambda_{12},\lambda_{13},\lambda_{14},\lambda_{21},\lambda_{23},\lambda_{24},\lambda_{31},\lambda_{32},\lambda_{34},\lambda_{41},\lambda_{42},\lambda_{43},\\
				&&\rho_{12},\rho_{13},\rho_{14},\rho_{21},\rho_{23},\rho_{24},\rho_{31},\rho_{32},\rho_{34},\rho_{41},\rho_{42},\rho_{43}.
			\end{eqnarray*}
			Since there are already relatively many relations of $G_4$, we will not list them all. Instead, we comment on the values that the function $\phi$ takes on them. The simplex $C_4$ is a tetrahderon. The faces of $C_4$ contain six edges, four triangles and one tetrahedron:
			\begin{eqnarray*}
				&&\Edg_4=\{\{1,2\},\{1,3\},\{1,4\},\{2,3\},\{2,4\},\{3,4\}\}\\
				&&\Tri_4=\{\{1,2,3\},\{1,2,4\},\{1,3,4\},\{2,3,4\}\}\\
				&&\Tet_4=\{\{1,2,3,4\}\}.
			\end{eqnarray*}
			The function $\phi$ takes the following values on the generators (see \Cref{figure:sautf4}):
			\begin{eqnarray*}
				&&\phi(\lambda_{12})=\phi(\lambda_{21})=\phi(\rho_{12})=\phi(\rho_{21})=\{1,2\},\\
				&&\phi(\lambda_{13})=\phi(\lambda_{31})=\phi(\rho_{13})=\phi(\rho_{31})=\{1,3\},\\
				&&\phi(\lambda_{14})=\phi(\lambda_{41})=\phi(\rho_{14})=\phi(\rho_{41})=\{1,4\},\\
				&&\phi(\lambda_{23})=\phi(\lambda_{32})=\phi(\rho_{23})=\phi(\rho_{32})=\{2,3\},\\
				&&\phi(\lambda_{24})=\phi(\lambda_{42})=\phi(\rho_{24})=\phi(\rho_{42})=\{2,4\},\\
				&&\phi(\lambda_{34})=\phi(\lambda_{43})=\phi(\rho_{34})=\phi(\rho_{43})=\{3,4\}.
			\end{eqnarray*}
			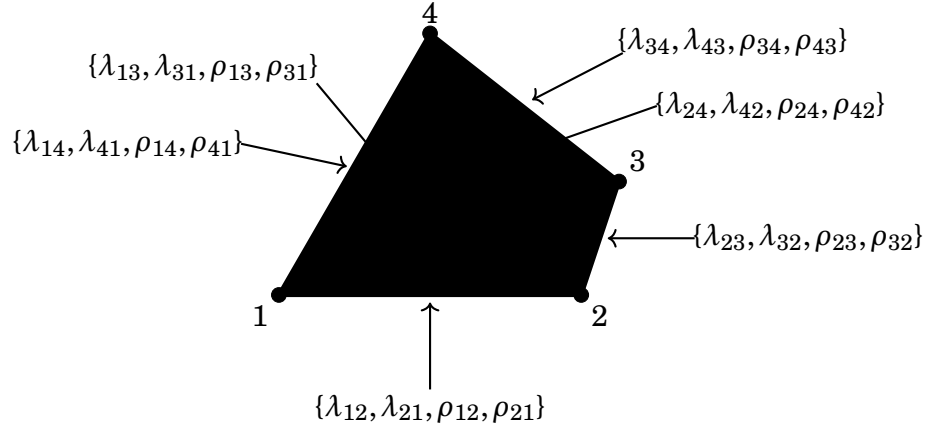
\begin{figure}[h]
				\centering
				\begin{tikzpicture}
					\coordinate (A) at (0,0);
					\coordinate (B) at (4,0);
					\coordinate (C) at (4.5,1.5);
					\coordinate (D) at (2,3.46);
					\coordinate (P) at (1.16,2.01);
					\draw[black, ultra thick] (A) -- (B);
					\draw[black, ultra thick, dashed] (A) -- (C);
					\draw[black, ultra thick] (B) -- (C);
					\draw[black, ultra thick] (A) -- (D);
					\draw[black, ultra thick] (B) -- (D);
					\draw[black, ultra thick] (C) -- (D);
					\filldraw[black] (A) circle (2.75pt) node[below left]{$1$};
					\filldraw[black] (B) circle (2.75pt) node[below right]{$2$};
					\filldraw[black] (C) circle (2.75pt) node[above right]{$3$};
					\filldraw[black] (D) circle (2.75pt) node[above]{$4$};
					\fill[black, opacity=0.1] (A) -- (B) -- (D) -- cycle;
					\fill[black, opacity=0.1] (B) -- (C) -- (D) -- cycle;
					\coordinate (e12) at (2,-1.5);
					\coordinate (e13) at (-1,3);
					\coordinate (e14) at (-2,2);
					\coordinate (e23) at (7,0.75);
					\coordinate (e24) at (6.5,2.5);
					\coordinate (e34) at (6,3.35);
					\node at (e12) {$\{\lambda_{12},\lambda_{21},\rho_{12},\rho_{21}\}$};
					\node at (e13) {$\{\lambda_{13},\lambda_{31},\rho_{13},\rho_{31}\}$};
					\node at (e14) {$\{\lambda_{14},\lambda_{41},\rho_{14},\rho_{41}\}$};
					\node at (e23) {$\{\lambda_{23},\lambda_{32},\rho_{23},\rho_{32}\}$};
					\node at (e24) {$\{\lambda_{24},\lambda_{42},\rho_{24},\rho_{42}\}$};
					\node at (e34) {$\{\lambda_{34},\lambda_{43},\rho_{34},\rho_{43}\}$};
					\draw[->,thick] (2,-1.25) -- (2,-0.1);
					\draw[thick] (0.5,2.8) -- (P);
					\draw[->,thick,dashed] (P) -- (2.2,0.8);
					\draw[->,thick] (-0.5,2) -- (0.9,1.7);
					\draw[->,thick] (5.5,0.75) -- (4.35,0.75);
					\draw[->,thick] (5,2.5) -- (3.1,1.83);
					\draw[->,thick] (4.55,3.2) -- (3.35,2.58);
				\end{tikzpicture}
				\caption{The simplex $C_4$ with the assignments of the generators of $\SAut(F_4)$ to its edges given by $\phi$.}
				\label{figure:sautf4}
		\end{figure}}
		
		\noindent
		\emph{Let $\leq i,j,k,l\leq 4$ be pairwise distinct. The values of $\phi$ on the relators are the following:
			\begin{eqnarray*}
				&&\phi([\lambda_{ij},\rho_{ij}])=\{i,j\},\quad\phi([\lambda_{ij},\lambda_{kl}])=\phi(\rho_{ij},\rho_{kl})=\phi([\lambda_{ij},\rho_{kl}])=\{1,2,3,4\},\\
				&&\phi([\lambda_{ij},\lambda_{kj}])=\phi([\rho_{ij},\rho_{kj}])=\phi([\lambda_{ij},\rho_{ik}])=\phi([\lambda_{i,j},\rho_{kj}])\\
				&&\qquad\qquad\quad\text{ }=\phi(\lambda_{ik}^{\pm}[\lambda_{jk}^{\pm},\lambda_{ij}^{-1}])=\phi(\rho_{ik}^{\pm}[\rho_{jk}^{\pm},\rho_{ij}^{-1}])\\
				&&\qquad\qquad\quad\text{ }=\phi(\lambda_{ik}^{\pm}[\rho_{jk}^{\mp},\lambda_{ij}])=\phi(\rho_{ik}^{\pm}[\lambda_{jk}^{\mp},\rho_{ij}])=\{i,j,k\}.
		\end{eqnarray*}}
	\end{example}
	
	\subsection{Sum of squares decomposition for square and opposite part}
	It turns out that the square and opposite parts behave well with respect to sum of squares decompositions. The precise statement is given below.
	\begin{lemma}\label{lemma:sq_op_sos}
		The matrices $\Sq^-_n$ and $\Op^+_n+2\Op^-_n$ are sums of squares.
	\end{lemma}
	\begin{proof}
		For any $1\leq i\neq j\leq n$, define $d^{i,j}\in\mathbb{M}_{|\altmathcal{S}_n|\times 1}(\RR G_n)$ to be the column vector indexed by the generators of $G_n$ with the only non-zero entries being $1-s$, indexed by the generators $s$ satisfying $\phi(s)=\{i,j\}$. That $\Sq^-_n$ is a sum of squares follows from the decomposition below:
		$$
		\Sq^-_n=\sum_{1\leq i\neq j\leq n}d^{i,j}\left(d^{i,j}\right)^*.
		$$
		
		We have already observed that both $\Op^+_n$ and $\Op^-_n$ (in this case, just by the definition) have vanishing $(s,t)$-entries whenever $\phi(s)\cup\phi(t)\in\Tri_n$ or the generator $t$ has the same indices as $s$ but with the reversed order. In addition, the diagonal vanishes for $\Op^-_n$ as well. This is not the case for $\Op^+_n$. However, since $\Op^+_n$ is a sum of squares by its definition, it follows that the same applies to its diagonal part. On the other hand, the non-diagonal parts of $\Op^+_n$ and $2\Op^-_n$ cancel out -- for distinct $1\leq i,j,k,l\leq n$ we have:
		\begin{align*}
			\left(\Op^+_n\right)_{\alpha_{ij},\beta_{kl}}&=\left(\frac{\partial[\alpha_{ij},\beta_{kl}]}{\partial \alpha_{ij}}\right)^*\frac{\partial[\alpha_{ij},\beta_{kl}]}{\partial \beta_{kl}}+\left(\frac{\partial[\beta_{kl},\alpha_{ij}]}{\partial \alpha_{ij}}\right)^*\frac{\partial[\beta_{kl},\alpha_{ij}]}{\partial \beta_{kl}}\\
			&=-2(1-\alpha_{ij})(1-\beta_{kl})^*.
		\end{align*} 
	\end{proof}
	\section{Symmetrization of the adjacent part}\label{section:symmetrization}
	In this section we show how to get the lower bound for spectral gap of $\Adj_n=\Adj_n^++\Adj_n^-$ for $G_n$ once we assume that the lower bound for the corresponding spectral gap for $G_m$ is known for some $n\geq m\geq 3$. Our method is inspired by the symmetrization method used in \cite{kkn}.
	
	\subsection{General setting for symmetrization}\label{subsection:general_setting}
	Note that the symmetric group $\Sym_k$ acts by automorphisms on $G_k$ by sending the generator $\omega_{i,j}$ to $\omega_{\sigma(i),\sigma(j)}$ for any permutation $\sigma\in \Sym_k$. This induces an action of $\Sym_k$ on $\RR G_k$ by automorphisms which preserves the $*$-involution. Using this action we can define the action on $\mathbb{M}_{|\altmathcal{S}_k|\times|\altmathcal{S}_k|}(\RR G_k)$ by the formula
	\[
	\left(\sigma A\right)_{s,t}=\sigma \left(A_{\sigma^{-1}(s),\sigma^{-1}(t)}\right),
	\]
	for any matrix $A$ over $\RR G_k$ indexed by the generators $\altmathcal{S}_k$.
	This is also the action by $*$-algebra automorphisms (the multiplication being the group ring matrix multiplication). Thus, for any matrix $A\in\mathbb{M}_{|\altmathcal{S}_k|\times|\altmathcal{S}_k|}(\RR G_k)$ and $\sigma\in \Sym_k$, the matrix $\sigma A$ is a sum of squares provided $A$ was a sum of squares.

	\subsection{Invariance of Laplacians}\label{subsection:inv}
	Let us show first that the adjacent Laplacians $\Adj_k^+$ and $\Adj^-_k$ for $G_k$ are $\Sym_k$-invariant. The statement for $\Adj^-_k$ is straightforward:
	\begin{align*}
		\left(\sigma\Adj_k^-\right)_{s,t}&=\sigma\left(\left(\Adj^-_k\right)_{\sigma^{-1}(s),\sigma^{-1}(t)}\right)=\sigma\left(\left(1-\sigma^{-1}(s)\right)\left(1-\sigma^{-1}(t)\right)^*\right)\\
		&=(1-s)(1-t)^*=\left(\Adj^-_k\right)_{s,t}.
	\end{align*}
	Let us show the invariance of $\Adj_k^+$. Let $a=|\altmathcal{S}_k|$ be the number of generators of $G_k$ and $b$ be the number of its relations $r$ such that $\phi(r)\in\Tri_k$. Suppose moreover that this set of relations is invariant with respect to the action of $\Sym_k$, that is, for all $\sigma\in \Sym_k$ and $r$ satisfying $\phi(r)\in\Tri_k$, the word $\sigma (r)$ is another relator in $\phi^{-1}(\Tri_k)$. Notice that each permutation $\sigma$ from $\Sym_k$ defines permutations of $\{1,\ldots,a\}$ and $\{1,\ldots,b\}$ (we fix the order of generators $\altmathcal{S}_k$ and relations $r$ such that $\phi(r)\in\Tri_k$). Having this, we can endow both products $\left(\RR G_k\right)^a$ and $\left(\RR G_k\right)^b$ with the action of $\Sym_k$ as follows:
	\begin{align*}
		\sigma\left(\xi_1,\ldots,\xi_a\right)&=\left(\sigma\xi_{\sigma^{-1}(1)},\ldots,\sigma\xi_{\sigma^{-1}(a)}\right),\\
		\sigma\left(\xi_1,\ldots,\xi_b\right)&=\left(\sigma\xi_{\sigma^{-1}(1)},\ldots,\sigma\xi_{\sigma^{-1}(b)}\right).
	\end{align*} 
	Finally, denote by $J$ the Jacobian map from $\left(\RR G_k\right)^a$ to $\left(\RR G_k\right)^b$ given by the relations $r$ such that $\phi(r)\in\Tri_k$, that is $\Adj_k^+=J^*J$. It turns out that the invariance of $\Adj_k^+$ follows from the $\Sym_k$-equivariance of $J$:
	\begin{lemma}\label{lemma:adj_invariance}
		For any $\sigma\in \Sym_k$ we have $\sigma\Adj_k^+=\Adj_k^+$ provided $J$ is equivariant.
	\end{lemma}
	\begin{proof}
		Suppose $J$ is equivariant. Let $\sigma\in \Sym_k$. Since we have ordered the generators in $\altmathcal{S}_k$, we can identify them with the corresponding indices from the set $\{1,\ldots,a\}$. Thus, we must show that $\left(\sigma\Adj_k^+\right)_{i,j}=\left(\Adj_k^+\right)_{i,j}$ for any $1\leq i,j\leq a$. 
		
		It is easy to check that the equivariance of $J$ is equivalent to the following condition: 
		$
		J_{y,\sigma^{-1}(x)}=\sigma^{-1}J_{\sigma(y),x}
		$
		for any $1\leq x\leq a$ and $1\leq y\leq b$. Using this relationship, we get
		\begin{align*}
			\left(\sigma\Adj_k^+\right)_{i,j}&=\sigma\left(\sum_{l}\left(J_{l,\sigma^{-1}(i)}\right)^*J_{l,\sigma^{-1}(j)}\right)=\sigma\left(\sum_{l}\sigma^{-1}\left(\left(J_{\sigma(l),i}\right)^*J_{\sigma(l),j}\right)\right)\\
			&=\sum_{l}\left(J_{\sigma(l),i}\right)^*J_{\sigma(l),j}=\left(\Adj_k^+\right)_{i,j}.
		\end{align*}
	\end{proof}
	The following lemma is the key tool to prove the equivariance of $J$.
	\begin{lemma}\label{lemma:auxiliaryJac}
		For any $\sigma\in \Sym_N$, $s\in\{s_1,\ldots,s_N\}$ and $r$ a word in the free group $F_N$ generated by $\{s_1,\ldots,s_N\}$, the following holds:
		$$
		\frac{\partial r}{\partial(\sigma (s))}=\sigma\left(\frac{\partial\left(\sigma^{-1}(r)\right)}{\partial s}\right),
		$$
		where the action of $\Sym_N$ on $F_N$ is given by $\tau(s_i)=s_{\tau(i)}$ for any $\tau\in\Sym_N$.
	\end{lemma}
	\begin{proof}
		We prove the assertion by induction on the word length of the relation $r$. Note first that both sides of the equation in question vanish for $r=1$. Suppose $r$ is a single generator $s'$ (if $r$ is the inverse of a generator, the proof is analogous). In the case $\sigma(s)\notin\{s',(s')^{-1}\}$ both sides of the equation in question vanish. If $\sigma(s)=s'$, then
		$$
		\frac{\partial r}{\partial(\sigma (s))}=\frac{\partial s'}{\partial s'}=1=\sigma\left(\frac{\partial s}{\partial s}\right)=\sigma\left(\frac{\partial\left(\sigma^{-1}(r)\right)}{\partial s}\right),
		$$
		while if $\sigma(s)=(s')^{-1}$, then
		$$
		\frac{\partial r}{\partial(\sigma (s))}=\frac{\partial s'}{\partial (s')^{-1}}=-s'=\sigma\left(\frac{\partial s^{-1}}{\partial s}\right)=\sigma\left(\frac{\partial\left(\sigma^{-1}(r)\right)}{\partial s}\right).
		$$
		Suppose now the assertion follows for all words of length $k\geq 1$. Take a word $r$ of length $k+1$. Then $r=uv$ for some words $u$ and $v$ of length at most $k$. Using the induction assumption, we get
		\begin{align*}
			\frac{\partial r}{\partial(\sigma (s))}&=\frac{\partial u}{\partial(\sigma (s))}+u\frac{\partial v}{\partial(\sigma (s))}\\
			&=\sigma\left(\frac{\partial\left(\sigma^{-1}(u)\right)}{\partial s}+\sigma^{-1}(u)\frac{\partial\left(\sigma^{-1}(v)\right)}{\partial s}\right)\\
			&=\sigma\left(\frac{\partial\left(\sigma^{-1}(uv)\right)}{\partial s}\right)=\sigma\left(\frac{\partial\left(\sigma^{-1}(r)\right)}{\partial s}\right).
		\end{align*} 
	\end{proof}
	\begin{corollary}\label{corollary:jacobianInvariance}
		The Jacobian homomorphism $J:(\mathbb{R}G)^a\rightarrow(\mathbb{R}G)^b$ is an equivariant map.
	\end{corollary}
	\begin{proof}
		By Lemma \ref{lemma:auxiliaryJac}, for any $1\leq x\leq a$ and $1\leq y\leq b$, we get 
		\begin{align*}
			J_{y,\sigma^{-1}(x)}=\frac{\partial (r_y)}{\partial\left(\sigma^{-1}(s_x)\right)}=\sigma^{-1}\left(\frac{\partial \left(\sigma (r_y)\right)}{\partial s_x}\right)=\sigma^{-1}J_{\sigma(y),x}.
		\end{align*}
	\end{proof}
	\subsection{Adjacent part symmetrizes well}
	In this section we show how to express $\Adj_n^{\pm}$ in terms of $\Adj_m^{\pm}$ and the action of the symmetric group $\Sym_n$, for any $n\geq m\geq 3$. This allows us to induce spectral gaps for $\Adj$ parts (cf. \Cref{corollary:main_ind}).
	
	For any $n\geq 3$ and $\theta\in\Tri_n$, we denote by $\Adj^{\pm}_{\theta}\in\mathbb{M}_{|\altmathcal{S}_n|\times|\altmathcal{S}_n|}(\RR G_n)$ the embedding of the adjacent part given by the traingle $\theta$ of $C_n$, i.e., for any generators $s$ and $t$ of $G_n$, we have
	\begin{align*}
		\left(\Adj_{\theta}^-\right)_{s,t}=\begin{cases}
			(1-s)(1-t)^* & \text{if } \phi(s)\cup\phi(t)=\theta \\
			0 & \text{if } \phi(s)\cup\phi(t)\neq\theta
		\end{cases},\quad
		\left(\Adj_{\theta}^+\right)_{s,t}=\sum_{\phi(r)=\theta}\left(\frac{\partial r}{\partial s}\right)^*\frac{\partial r}{\partial t}.
	\end{align*}
	Since it shall be clear from the context, given $m\leq n$, we shall also denote by $\Adj_m^{\pm}$ the canonical embedding of $\Adj_m^{\pm}\in\mathbb{M}_{|\altmathcal{S}_m|\times|\altmathcal{S}_m|}(\RR G_m)$ into $\mathbb{M}_{|\altmathcal{S}_n|\times|\altmathcal{S}_n|}(\RR G_n)$:
	\begin{align*}
		\left(\Adj_m^-\right)_{s,t}&=\begin{cases}
			(1-s)(1-t)^* & \text{if } \phi(s)\cup\phi(t)\subseteq\{1,\ldots,m\} \\
			0 & \text{if } \phi(s)\cup\phi(t)\nsubseteq\{1,\ldots,m\}
		\end{cases},\\
		\left(\Adj_m^+\right)_{s,t}&=\sum_{\phi(r)\subseteq\{1,\ldots,m\}}\left(\frac{\partial r}{\partial s}\right)^*\frac{\partial r}{\partial t}.
	\end{align*}
	Due to the same reason, we decided to remove the grading $n$ from the notation $\Adj_{\theta}^{\pm}$.
	\begin{lemma}\label{lemma:nice_embeddings}
		For $n\geq m\geq 3$, one has
		$$
		\sum_{\sigma\in\Sym_n}\sigma\left(\Adj_m^{\pm}\right)=m(m-1)(m-2)(n-3)!\Adj_n^{\pm}.
		$$ 
	\end{lemma}
	\begin{proof}
		The idea of the proof is to transfer the action of the symmetric group $\Sym_n$ from the matrices over group rings to the simplex $C_n$ and apply the invariance of $\Adj_m^{\pm}$ under the action of $\Sym_m$. The transfer is possible due to the following relationships:
		\begin{equation}\label{equation:simplex_transfer}
			\sigma\left(\Adj_{\theta}^{\pm}\right)=\Adj_{\sigma(\theta)}^{\pm},
		\end{equation}
		holding for every permutation $\sigma\in\Sym_n$ and a triangle $\theta\in\Tri_n$. Let us show (\ref{equation:simplex_transfer}) for $\Adj^+$ only since the proof for $\Adj^-$ is a direct application of its definition and the action of $\Sym_n$ on $\mathbb{M}_{|\altmathcal{S}_n|\times|\altmathcal{S}_n|}(\RR G_n)$. Pick two generators $s$ and $t$ of $G_n$. Then, it follows by \Cref{lemma:auxiliaryJac} that
		\begin{align*}
			\left(\sigma\left(\Adj_{\theta}^+\right)\right)_{s,t}
			&=\sigma\left(\sum_{\phi(r)=\theta}\left(\frac{\partial r}{\partial\sigma^{-1}(s)}\right)^*\frac{\partial r}{\partial\sigma^{-1}(t)}\right)\\
			&=\sum_{\phi(r)=\sigma(\theta)}\left(\frac{\partial r}{\partial s}\right)^*\frac{\partial r}{\partial t}=\Adj_{\sigma(\theta)}^+.
		\end{align*}
		Note that, for any $k\geq 3$, we can express $\Adj_k^{\pm}$ by summing its projections to the triangles of $C_k$:
		\begin{align*}
			\Adj_k^{\pm}=\sum_{\theta\in\Tri_k}\Adj_{\theta}^{\pm}.
		\end{align*}
		Applying (\ref{equation:simplex_transfer}) and the orbit-stabilizer theorem applied to the action of $\Sym_k$ on the triangles $\Tri_k$, we get
		\begin{align}\label{align:genaral_adj}
			\Adj_k^{\pm}=\sum_{\sigma\in\Sym_k}\frac{1}{|\Stab\{1,2,3\}|}\Adj_{\sigma(\{1,2,3\})}^{\pm}=\frac{1}{(k-3)!}\sum_{\sigma\in\Sym_k}\sigma\left(\Adj_3^{\pm}\right).
		\end{align}
		Let $\tau_i\in\Sym_n$, $i=1,\ldots,\frac{n!}{m!}$ be the representatives of the cosets $\Sym_n/\Sym_m$. Applying (\ref{align:genaral_adj}) first for $k=n$ and then for $k=m$, we get
		\begin{align*}
			(n-3)!\Adj^{\pm}_n&=\sum_{\sigma\in\Sym_n}\sigma\left(\Adj^{\pm}_3\right)\\
			&=\sum_{i}\tau_i\sum_{\tau\in\Sym_m}\tau\left(\Adj^{\pm}_3\right)\\
			&=(m-3)!\sum_{i}\tau_i\Adj^{\pm}_m.
		\end{align*}
		We conclude the proof by applying $\Sym_m$-invariance of $\Adj^{\pm}_m$ (cf. \Cref{lemma:adj_invariance} and \Cref{corollary:jacobianInvariance}):
		\begin{align*}
			(n-3)!\Adj^{\pm}_n&=(m-3)!\sum_{i}\tau_i\Adj^{\pm}_m\\
			&=(m-3)!\sum_{i}\tau_i\frac{1}{m!}\sum_{\sigma\in\Sym_m}\sigma\left(\Adj^{\pm}_m\right)\\
			&=\frac{1}{m(m-1)(m-2)}\sum_{\sigma\in\Sym_n}\sigma\left(\Adj^{\pm}_m\right).
		\end{align*}
	\end{proof}
	For any $k\geq 3$, let $\Adj_k=\Adj_k^++\Adj_k^-$. We have the following
	\begin{corollary}[cf. \Cref{theorem:main0}]\label{corollary:main_ind}
		The matrix 
		\begin{align*}
			\Adj_n-\frac{n-2}{m-2}\lambda I_{|\altmathcal{S}_n|}
		\end{align*}
		is a sum of squares in $\mathbb{M}_{|\altmathcal{S}_n|\times|\altmathcal{S}_n|}(\RR G_n)$, provided $\Adj_m-\lambda I_{|\altmathcal{S}_m|}$ is a sum of squares in $\mathbb{M}_{|\altmathcal{S}_m|\times|\altmathcal{S}_m|}(\RR G_m)$.
	\end{corollary}
	\begin{proof}
		As in the case of $\Adj^{\pm}_m$, let us also use the notation $I_{|\altmathcal{S}_m|}$ for the canonical embedding of $I_{|\altmathcal{S}_m|}\in\mathbb{M}_{|\altmathcal{S}_m|\times|\altmathcal{S}_m|}(\RR G_m)$ into $\mathbb{M}_{|\altmathcal{S}_n|\times|\altmathcal{S}_n|}(\RR G_n)$. We show first that
		\begin{align}\label{eqnarray:identity_sym}
			(n-2)!m(m-1)I_{|\altmathcal{S}_n|}=\sum_{\sigma\in \Sym_n}\sigma I_{|\altmathcal{S}_m|}.
		\end{align}
		Note that the right-hand side of the expression above is a multiple of the identity matrix. Let $\omega_0$ be $1$ in the case $G_n=\SL_n(\ZZ)$ and $2$ in the case $G_n=\SAut(F_n)$. Then each summand $\sigma I_{|\altmathcal{S}_m|}$ contributes $|\altmathcal{S}_m|=m(m-1)\omega_0$ neutral elements to the diagonal which yields $n!m(m-1)\omega_0$ neutral elements altogether. Since the diagonal of $I_{|\altmathcal{S}_n|}$ has $|\altmathcal{S}_n|=n(n-1)\omega_0$ neutral elements, the multiplication factor is equal to $(n-2)!m(m-1)$.
		
		Suppose $\Adj_m-\lambda I_{|\altmathcal{S}_m|}$ is a sum of squares. Combining Lemma \Cref{lemma:nice_embeddings} and (\ref{eqnarray:identity_sym}), we conclude that the following matrix is a sum of squares in $\mathbb{M}_{|\altmathcal{S}_n|\times|\altmathcal{S}_n|}(\RR G_n)$:
		$$
		m(m-1)(m-2)(n-3)!\Adj_n-(n-2)!m(m-1)\lambda I_{|\altmathcal{S}_n|}.
		$$
		Dividing the expression above by $(n-3)!m(m-1)$, we get the assertion.
	\end{proof}
	\section{Application to $G_n=\SL_n(\ZZ)$}\label{section:application_sln}
	In this section, let $G_n=\SL_n(\ZZ)$ for any $n\geq 3$. Using the algorithm for estimating lower bounds for spectral gaps for matrices over group rings described in \cite{kmn}, we were able to obtain the following 
	\begin{lemma}\label{lemma:sl3_computation}
		The matrix $\Adj_3-0.217I_{6}$ is a sum of squares in $\mathbb{M}_{6\times 6}(\RR \SL_3(\ZZ))$.
	\end{lemma}
	\noindent The replication details of the computation justifying the result above have been desribed in subsection \ref{subsection:replication}.
	
	\begin{corollary}[cf. \Cref{corollary:main_intr}]\label{corollary:main}
		Let $\Delta_1$ be the cohomological Laplacian in degree one of $G_n$. Then $\Delta_1-\lambda I_{|\altmathcal{S}_n|}$ is a sum of squares in $\mathbb{M}_{|\altmathcal{S}_n|\times|\altmathcal{S}_n|}(\RR G_n)$ for $\lambda=0.217(n-2)$.
	\end{corollary}
	\begin{proof}
		\Cref{lemma:sl3_computation} together with \Cref{corollary:main_ind} show that $\Adj_n-0.217(n-2)I_{|\altmathcal{S}_n|}$ is a sum of squares in $\mathbb{M}_{|\altmathcal{S}_n|\times|\altmathcal{S}_n|}(\RR G_n)$. It follows by \Cref{lemma:sq_op_sos} that the matrices $\Sq_n^-$ and $\widetilde{\Op}_n=\Op_n^++2\Op_n^-$ are sums of squares. From this we conclude that $\Delta_1-0.217(n-2)I_{|\altmathcal{S}_n|}$ is a sum of squares in $\mathbb{M}_{|\altmathcal{S}_n|\times|\altmathcal{S}_n|}(\RR G_n)$:
		$$
		\Delta_1-0.217(n-2)I_{|\altmathcal{S}_n|}=\frac{1}{2}\left(2\left(\Adj_n-0.217(n-2)I_{|\altmathcal{S}_n|}\right)+\widetilde{\Op}_n+\Op_n^++2\Sq_n^-\right).
		$$
	\end{proof}
	\begin{remark}\emph{
		\Cref{corollary:main} would actually generalize straightforward to $\SAut(F_n)$, provided we could perform a base computation for $\Adj_n$ e.g. for $n=4$ (cf. \Cref{lemma:sl3_computation}) -- we would need to add $\Sq_n^+$ to get $\Delta_1$. Since $\Sq_n^+$ is a sum of squares, that would not be a problem. Unfortunately, we were unable to obtain such a base result for $\SAut(F_n)$.
	}
	\end{remark}
	\subsection{Replication details for $\SL_3(\mathbb{Z})$}\label{subsection:replication}
	The arguments for converting the numerical sum of squares approximation to the actual one have been described in subsection 3.2 of \cite{kmn}. Further details concerning finding the numerical approximation can be also found in \cite{kmn}. 
	
	We provide a Julia \cite{bezanson2017julia} code to replicate the result of \Cref{lemma:sl3_computation}. The replication has been described at \cite{Kaluba_LowCohomologySOS}, in the "README.md" file, in the section marked with the arXiv identifier of this preprint. Note that, in order to access the version of the repository corresponding to this article, switching to the branch marked with its arXiv identifier is necessary. The script "SL\textunderscore3\textunderscore Z\textunderscore adj.jl" contains the replication code for the result of \Cref{lemma:sl3_computation}. We encourage, however, to use the precomputed solution and run the script "SL\textunderscore3\textunderscore Z\textunderscore adj\textunderscore cert.jl" contained in the "sl3\textunderscore  adj\textunderscore precom" folder. It should take approximately $1$ minute on a standard laptop computer to run the certification script providing the rigorous mathematical proof (the script containing the whole computation takes in turn about $2$ hours). The hermitian-squared matrices from the decomposition obtained in \Cref{lemma:sl3_computation} have entries supported on the ball of radius $2$ with respect to the word-length metric defined by the generating set $\altmathcal{S}_3$ consisting of six elementary matrices $E_{12}$, $E_{13}$, $E_{21}$, $E_{23}$, $E_{31}$, and $E_{32}$. 

	\vspace{20pt} 
\end{document}